\documentclass[11pt]{amsart}
\usepackage{geometry} % see geometry.pdf on how to lay out the page. There's lots.
\usepackage{multirow}

\usepackage{mathrsfs}
\usepackage[french]{babel}
\usepackage[latin1]{inputenc}

\usepackage{amssymb}
\usepackage{pdfsync}
\usepackage{array}
\usepackage{amsmath}
\usepackage{amsfonts}
\usepackage{amstext}
\usepackage{amsopn}
\usepackage{amsbsy}
\usepackage{shortvrb,fancybox}
\usepackage{rotating}
\usepackage{xy}
\xyoption{all}
\usepackage{multirow}

\newcommand{\Z}{\mathbb{Z}}
\newcommand{\Q}{\mathbb{Q}}

\newcommand{\C}{\mathbb{C}}

\newcommand{\aK}{\Overline{2.5}{-1.5}{K}}

\newcommand{\Kl}{K^{\lambda}}

\newcommand{\q}{\mathfrak{q}}
\newcommand{\Pcq}{\mathcal{P}^{\mathfrak{q}}}
\newcommand{\OLq}{\mathcal{O}_{L^{\mathfrak{q}}}}

\newcommand{\Kq}{K_{\mathfrak{q}}}

\newcommand{\kq}{k_{\mathfrak{q}}}
\newcommand{\akq}{\overline{k_{\mathfrak{q}}}}
\newcommand{\Iq}{I_{\mathfrak{q}}}
\newcommand{\Dq}{D_{\mathfrak{q}}}

\newcommand{\sq}{\sigma_{\mathfrak{q}}}

\newcommand{\Lq}{L^{\mathfrak{q}}}
\newcommand{\bq}{\beta_{\mathfrak{q}}}

\newcommand{\Pq}{P_{\mathfrak{q}}}

\newcommand{\phip}{\varphi_p}

\newcommand{\p}{\mathfrak{p}}

\newcommand{\Qp}{\mathbb{Q}_{p}}
\newcommand{\Kp}{K_{\mathfrak{p}}}

\newcommand{\aKp}{\overline{K_{\mathfrak{p}}}}

\newcommand{\Fp}{\mathbb{F}_{p}}
\newcommand{\Fpx}{\mathbb{F}_{p}^{\times}}

\newcommand{\Ip}{I_{\mathfrak{p}}}
\newcommand{\ep}{e_{\mathfrak{p}}}
\newcommand{\rp}{r_{\mathfrak{p}}}

\newcommand{\ap}{a_{\p}}

\newcommand{\GL}{\mathrm{GL}}
\newcommand{\PGL}{\mathrm{PGL}}
\newcommand{\SL}{\mathrm{SL}}

\newcommand{\la}{\lambda}
\newcommand{\lad}{\lambda^{12}}

\newcommand{\kip}{\chi_p}

\newcommand{\Mod}{\mkern-5mu \mod}

\def\Overline #1#2#3%
{\mkern #1mu
\overline {\mkern-#1mu #3 \mkern-#2mu}
\mkern #2mu }

\newtheorem*{TheoSerre}{Th\'eor\`eme (Serre, \cite{[Se]})}

\newtheorem*{TheoHommodpA}{Th\'eor\`eme}
\newtheorem*{TheoHommodpAQ}{Th\'eor\`eme pour $\Q$}
\newtheorem*{TheoHommodpA_Irr1}{Th\'eor\`eme dans le cas irr\'eductible (I)}
\newtheorem*{TheoHommodpA_Irr2}{Th\'eor\`eme dans le cas irr\'eductible (II)}
\newtheorem*{TheoHommodpA_Red}{Th\'eor\`eme dans le cas r\'eductible}

\newtheorem*{PropExcep}{Proposition}

\newtheorem*{Ssgrpesmax}{Sous-groupes maximaux de $\GL_2(\Fp)$}
\newtheorem*{Coro}{Corollaire}

\theoremstyle{plain}
\newtheorem{theoreme}{Th\'eor\`eme}[section]
\newtheorem{proposition}[theoreme]{Proposition}
\newtheorem{corollaire}[theoreme]{Corollaire}
\newtheorem{lemme}[theoreme]{Lemme}
\theoremstyle{definition}

\newtheorem{remarque}[theoreme]{Remarque}

\newtheorem{notations}[theoreme]{Notations}
\newtheorem{hypotheses}[theoreme]{Hypoth\`eses}

\AtBeginDocument{}

\DeclareMathAlphabet{\mathpzc}{OT1}{pzc}{m}{it}

\title[Borne uniforme pour les homoth\'eties]{Borne uniforme pour les homoth\'eties dans l'image de Galois associ\'ee aux courbes elliptiques}
\author[Agn\`es David]{Agn\`es David}
\address{\'ENS de Lyon -- UMPA \\
CNRS -- UMR 5669 \\
46 all\'ee d'Italie \\
69364 Lyon Cedex 07,
France
}
\email{Agnes.David@ens-lyon.org}
\begin{document}

\begin{abstract}
Let~$K$ be a fixed number field and~$G_K$ its absolute Galois group. We give a bound~$C(K)$, depending only on the degree, the class number and the discriminant of~$K$, such that for any elliptic curve~$E$ defined over~$K$ and any prime number~$p$ strictly larger than~$C(K)$, the image of the representation of~$G_K$ attached to the $p$-torsion points of~$E$ contains a subgroup of homotheties of index smaller than~$12$.
\end{abstract}

\maketitle
\tableofcontents

\section*{Introduction}

L'objet de cet article est de d\'emontrer un r\'esultat asymptotique et uniforme sur la taille du sous-groupe des homoth\'eties contenu dans l'image de la repr\'esentation associ\'ee aux points de torsion d'une courbe elliptique.

Le cadre pr\'ecis est le suivant.
On fixe un corps de nombres~$K$~; on en fixe \'egalement une cl\^oture alg\'ebrique~$\aK$ et on note~$G_K$ le groupe de Galois de~$\aK$ sur~$K$.
On d\'esigne par~$d$ le degr\'e de~$K$ sur~$\Q$ et~$h$ son nombre de classes d'id\'eaux.
On fixe une courbe elliptique~$E$ d\'efinie sur~$K$ et~$p$ un nombre premier sup\'erieur ou \'egal \`a~$5$.
On note~$E_p$ l'ensemble des points de~$E$ dans~$\aK$ qui sont de~$p$-torsion~;  c'est un espace vectoriel de dimension~$2$ sur~$\Fp$, sur lequel le groupe de Galois absolu~$G_K$ de~$K$ agit~$\Fp$-lin\'eairement.
On d\'esigne par~$\phip$ la repr\'esentation de~$G_K$ ainsi obtenue~; elle prend ses valeurs dans le groupe~$\GL(E_p)$ qui, apr\`es choix d'une base pour~$E_p$, est isomorphe \`a~$\GL_2(\Fp)$~; on note~$G_p$ l'image dans~$\GL(E_p)$ de la repr\'esentation~$\phip$.

Le groupe~$G_p$ est isomorphe au groupe de Galois de l'extension du corps~$K$ obtenue en lui adjoignant les coordonn\'ees des points de~$p$-torsion de~$E$ dans~$\aK$.
Cette extension est \og asymptotiquement grosse \fg\ au sens du th\'eor\`eme suivant.
\begin{TheoSerre}
 On suppose que la courbe $E$ n'a pas de multiplication complexe. Alors il existe une borne $C(K,E)$, ne d\'ependant que de $K$ et de $E$ et telle que pour tout $p$ strictement sup\'erieur \`a $C(K,E)$, la repr\'esentation $\phip$ est surjective.
\end{TheoSerre}
Des formes explicites pour la borne~$C(K,E)$ ont \'et\'e \'etablies par Masser et W\"ustholz dans \cite{[MW]} (voir \'egalement \cite{[Pel]}).

La question, pos\'ee dans~\cite{[Se]}, d'\'eliminer la d\'ependance en la courbe elliptique~$E$ dans la borne~$C(K,E)$, pour obtenir une version uniforme du th\'eor\`eme, s'est r\'ev\'el\'ee ardue.
D'importants r\'esultats ont n\'eanmoins \'et\'e obtenus dans cette direction lorsque le corps de base est celui des rationnels~$\Q$.
Mazur a ainsi montr\'e dans~\cite{[Maz]} que, si~$p$ n'est pas dans l'ensemble~$\{2, 3, 5, 7, 13, 11, 17, 19, 37, 43, 67, 163 \}$, alors la repr\'esentation~$\phip$ est irr\'eductible.
Plus r\'ecemment, Bilu et Parent (\cite{[PaBi]}) ont \'etabli l'existence  d'une borne absolue pour les nombres premiers~$p$ pour lesquels $G_p$ est inclus dans le normalisateur d'un sous-groupe de Cartan d\'eploy\'e (voir partie \ref{sssec:descr_Cdep} pour la d\'efinition), lorsque~$E$ n'a pas de multiplication complexe.
Dans le domaine des r\'esultats uniformes sur les points de torsion de courbes elliptiques, on mentionne \'egalement  (et on utilisera dans la partie~\ref{sec:red}) les travaux de Merel (\cite{[Mer96]}) et Parent (\cite{[Pa]}), qui permettent de borner l'ordre d'un point de torsion de~$E$ en fonction du degr\'e du corps de nombres sur lequel il est d\'efini.

Sans condition sur l'anneau des endomorphismes de la courbe~$E$, on peut prouver l'existence une borne~$C'(K,E)$ telle que, si~$p$ est strictement sup\'erieur \`a~$C'(K,E)$, alors l'image~$G_p$ de~$\phip$ contient les carr\'es des homoth\'eties de~$\GL(E_p)$. Le but de cet article est d'obtenir une version uniforme de ce r\'esultat.

Des questions semblables (asymptotiques et uniformes) sur les homoth\'eties ont \'et\'e trait\'ees par Eckstein dans sa th\`ese~(\cite{[Eck05]}), y compris pour la repr\'esentation associ\'ee au module de Tate de la courbe~$E$ et pour des vari\'et\'es ab\'eliennes \`a multiplications complexes de dimension quelconque.
N\'eanmoins, un cas, o\`u la repr\'esentation~$\phip$ est r\'eductible,  n'est pas couvert par le r\'esultat \'enonc\'e dans~\cite{[Eck05]} (voir th\'eor\`eme~18 ou \S 2.6.3).
Par ailleurs, les m\'ethodes employ\'ees ici diff\`erent en grande partie de celles de~\cite{[Eck05]}.
Les th\'eor\`emes ci-dessous apportent ainsi une r\'eponse compl\`ete et, dans certains cas, plus pr\'ecise \`a la question des homoth\'eties contenues dans l'image de la repr\'esentation~$\phip$.

\begin{TheoHommodpA}
On suppose que le corps~$K$ est diff\'erent de~$\Q$ et que le nombre premier~$p$ est non ramifi\'e dans~$K$.
\begin{itemize}
\item[$\bullet$] On suppose la repr\'esentation~$\phip$ irr\'eductible. Si~$p$ est sup\'erieur ou \'egal \`a~$17$, alors~$G_p$ contient les carr\'es des homoth\'eties. Si de plus~$p$ est congru \`a~$1$ modulo~$4$, alors~$G_p$ contient toutes les homoth\'eties.
\item[$\bullet$] On suppose la repr\'esentation $\phip$ r\'eductible. Si $p$ est strictement sup\'erieur \`a $(1 + 3^{6dh})^2$, alors $G_p$ contient un sous-groupe d'indice divisant~$8$ ou~$12$ des homoth\'eties.
\end{itemize}
\end{TheoHommodpA}

Dans un article \`a venir (\cite{[Dav10]}), on donnera un r\'esultat plus pr\'ecis pour le cas o\`u la repr\'esentation $\phip$ est r\'eductible~: en poursuivant les m\'ethodes de \cite{[Mom]} (et quitte \`a augmenter la borne inf\'erieure pour~$p$), on montrera que~$G_p$ contient tous les carr\'es des homoth\'eties, sauf dans un cas pr\'ecis o\`u le corps~$K$ contient le corps de classes de Hilbert d'un corps quadratique imaginaire et o\`u la repr\'esentation~$\phip$ pr\'esente de grandes similarit\'es avec celle qui proviendrait d'une courbe elliptique ayant des multiplications complexes.

Dans le th\'eor\`eme ci-dessus, on ne suppose~$K$ diff\'erent de~$\Q$ que pour utiliser, dans le cas d'une repr\'esentation~$\phip$ r\'eductible, les bornes sur l'ordre des points de torsion de~$E$ qui figurent dans \cite{[Mer96]} ou \cite{[Pa]}.
Lorsque le corps de base est~$\Q$, les m\^eme m\'ethodes, associ\'ees au r\'esultat de Mazur (\cite{[Maz]}) sur l'irr\'eductibilit\'e de la repr\'esentation~$\phip$, donnent l'\'enonc\'e suivant.
\begin{TheoHommodpAQ}
On suppose que le corps $K$ est \'egal \`a $\Q$. Si $p$ est sup\'erieur ou \'egal \`a $23$ et diff\'erent de $37$, $43$, $67$ et $163$, alors~$G_p$ contient tous les carr\'es des homoth\'eties.
Si de plus $p$ est congru \`a $1$ modulo $4$, alors $G_p$ contient toutes les homoth\'eties.
\end{TheoHommodpAQ}

Au cours de la d\'emonstration du th\'eor\`eme dans le cas irr\'eductible (partie \ref{subsec:excep}), on a explicit\'e la proposition suivante, g\'en\'eralisation \`a un corps de base~$K$ quelconque du lemme 18 (\S 8.4) de~\cite{[CheEff]} valable sur~$\Q$.

\begin{PropExcep}\label{propintro:excep}
On suppose que $p$ est strictement sup\'erieur \`a $20d + 1$. Alors~$G_p$ n'est pas un sous-groupe exceptionnel.
\end{PropExcep}

Le th\'eor\`eme a par ailleurs  comme cons\'equence la borne inf\'erieure uniforme suivante pour le nombre de points qui sont \`a la fois multiples et conjugu\'es galoisiens d'un m\^eme point d'ordre $p$ (la formulation de ce corollaire est inspir\'ee de la partie 3 de~\cite{[Lan65]})~:
\begin{Coro}
On suppose~$p$ non ramifi\'e dans~$K$ et strictement sup\'erieur \`a~$(1 + 3^{6dh})^2$.
Soit~$P$ un point de~$E$ dans~$\aK$, d'ordre $p$.
Alors l'ensemble des points du sous-groupe de~$E_p$ engendr\'e par~$P$ qui sont aussi dans l'orbite de~$P$ sous l'action de~$G_K$ est de cardinal sup\'erieur ou \'egal \`a~$\frac{p-1}{12}$.
\end{Coro}

L'\'equivalent du th\'eor\`eme ci-dessus pour la repr\'esentation de~$G_K$ associ\'ee au module de Tate de la courbe~$E$ fournirait une borne inf\'erieure uniforme semblable pour tous les points dont l'ordre est une puissance de~$p$.
Eckstein a obtenu de tels r\'esultats lorsque la repr\'esentation~$\phip$ est irr\'eductible et dans certains cas o\`u elle est r\'eductible (th\'eor\`eme~18 de~\cite{[Eck05]}~; voir \'egalement~\cite{[JPW]} pour une version non uniforme).

On distingue pour la d\'emonstration du th\'eor\`eme les diff\'erents  types de sous-groupes maximaux de~$\GL(E_p)$ pouvant contenir l'image de la repr\'esentation~$\phip$. La liste et les propri\'et\'es de ces sous-groupes, \'etablie dans  \cite{[Se]}, est rappel\'ee dans la partie \ref{sec:Ssgrpesmax}.

Dans la partie~\ref{sec:irr}, on traite le cas o\`u la repr\'esentation~$\phip$ est irr\'eductible. On y d\'emontre (partie~\ref{subsec:excep}) la proposition \'enonc\'ee ci-dessus. Lorsque l'image~$G_p$ de~$\phip$ est incluse dans le normalisateur d'un sous-groupe de Cartan (partie~\ref{subsec:NCartan}), on constate que la structure d'un tel groupe, alli\'ee \`a une hypoth\`ese de surjectivit\'e du d\'eterminant de~$\phip$, suffit \`a  assurer la pr\'esence dans~$G_p$ des carr\'es des homoth\'eties.

Dans la partie~\ref{sec:red}, on traite le cas o\`u la repr\'esentation~$\phip$ est r\'eductible. On s'inspire pour cela de l'\'etude du caract\`ere d'isog\'enie men\'ee par Mazur (\cite{[Maz]}) puis Momose (\cite{[Mom]}).
Apr\`es un rappel de ses propri\'et\'es locales (parties \ref{ssec:Inerenp} et \ref{subsec:horsp}), on observe, gr\^ace \`a la th\'eorie du corps de classes, que l'obstacle principal \`a la pr\'esence d'un gros sous-groupe des homoth\'eties dans~$G_p$ est l'existence d'un point de la courbe~$E$ d\'efini sur une \og petite \fg\  extension du corps de base~$K$. Les nombres premiers~$p$ pour lesquels une telle situation survient sont alors contr\^ol\'es par les bornes uniformes (\cite{[Mer96]}, \cite{[Pa]}) pour l'ordre des points de torsion des courbes elliptiques en fonction du degr\'e du corps de base.

Je tiens \`a remercier ici Jean-Pierre Wintenberger, qui a initi\'e et encadr\'e ce travail, ainsi que John Boxall pour m'avoir indiqu\'e les arguments de la partie~\ref{subsec:NCartan} et les fructueuses discussions lors de la r\'edaction de cet article.

\section{Sous-groupes maximaux de $\GL_2(\Fp)$}\label{sec:Ssgrpesmax}

Soit $V$ un espace vectoriel de dimension $2$ sur $\Fp$. 
Serre a d\'ecrit dans \cite{[Se]} (\S2) les sous-groupes maximaux du groupe $\GL(V)$ (qui est isomorphe \`a $\GL_2(\Fp)$)~;
pour la commodit\'e du lecteur, on rappelle ici les diff\'erents types de tels sous-groupes et celles de leurs propri\'et\'es qui serviront dans la suite du texte.

\subsection{Sous-groupes exceptionnels}

Un sous-groupe de~$\GL(V)$ est dit exceptionnel~(\cite{[Se]} \S 2.5) si son image dans~$\PGL(V)$ est isomorphe \`a l'un des groupes altern\'es ou sym\'etrique~$\mathfrak{A}_4$,~$\mathfrak{S}_4$ ou~$\mathfrak{A}_5$. On remarque que cette image est alors form\'ee d'\'el\'ements d'ordre inf\'erieur ou \'egal \`a~$5$.

\subsection{Sous-groupes de Cartan et leurs normalisateurs}

\subsubsection{D\'eploy\'e}\label{sssec:descr_Cdep}

Soient $D_1$ et $D_2$ deux droites distinctes  de~$V$. 
Les \'el\'ements de~$\GL(V)$ qui fixent (globalement)~$D_1$ d'une part et $D_2$ d'autre part forment un sous-groupe~$C_d$ de~$\GL(V)$, appel\'e sous-groupe de Cartan d\'eploy\'e (associ\'e \`a~$D_1$ et~$D_2$). Un tel sous-groupe est d'indice~$2$ dans son normalisateur~$N(C_d)$~; les \'el\'ements de~$N(C_d)$ qui ne sont pas dans $C_d$ sont alors tous les \'el\'ements de $\GL(V)$ qui \'echangent $D_1$ et $D_2$~(\cite{[Se]} \S 2.1.a et \S2.2).

Dans une base de $V$ dont les vecteurs engendrent $D_1$ et $D_2$, $C_d$ est form\'e des matrices diagonales~; il est ab\'elien de type $(p-1,p-1)$~; dans cette m\^eme base, les \'el\'ements de $N(C_d)\backslash C_d$ sont les matrices anti-diagonales.

L'ensemble des \'el\'ements de $C_d$ qui op\'erent trivialement (point par point) sur $D_1$ forme un sous-groupe, appel\'e demi-sous-groupe de Cartan d\'eploy\'e. Un tel sous-groupe est cyclique, d'ordre $p-1$~; dans une base choisie comme pr\'ec\'edemment, il est form\'e des matrices diagonales dont le premier terme diagonal vaut $1$. 

\subsubsection{Non d\'eploy\'e}\label{sssec:descr_Cnondep}

Soit $A$ une sous-$\Fp$-alg\`ebre de~$\mathrm{End}(V)$ qui est un corps \`a $p^2$ \'el\'ements. Alors le groupe multiplicatif $C_{nd}$ de $A$ est un sous-groupe de~$\GL(V)$, appel\'e sous-groupe de Cartan non d\'eploy\'e. Un tel sous-groupe est cyclique et de cardinal~$p^2 - 1$~; il est d'indice $2$ dans son normalisateur $N(C_{nd})$ et les \'el\'ements de~$N(C_{nd}) \backslash C_{nd}$ sont les \'el\'ements de~$\GL(V)$ qui agissent par conjugaison sur $C_{nd}$ par \'el\'evation \`a la puissance~$p$~(\cite{[Se]} \S 2.1.b et \S2.2).

En termes matriciels, il existe un \'el\'ement $\alpha$ non carr\'e de $\Fpx$ et une base de $V$ dans laquelle $C_{nd}$ est form\'e des matrices~:
$$
\left\{
\begin{pmatrix}
a & b \alpha\\
b & a \\
\end{pmatrix},
(a,b) \in \Fp^2 \backslash \{(0,0)\}
\right\}.
$$
Dans cette m\^eme base, les \'el\'ements de $N(C_{nd}) \backslash C_{nd}$ sont alors les matrices~:
$$
\left\{
\begin{pmatrix}
a & - b \alpha\\
b & - a \\
\end{pmatrix},
(a,b) \in \Fp^2 \backslash \{(0,0)\}
\right\}.
$$

\subsection{Sous-groupe de Borel}

Soit $D$ une droite de $V$. Les \'el\'ements de $\GL(V)$ qui stabilisent $D$ forment un sous-groupe de  $\GL(V)$, appel\'e sous-groupe de Borel. Un tel sous-groupe est de cardinal $p(p-1)^2$~; dans une base de $V$ dont le premier vecteur engendre $D$, il est constitu\'e des matrices triangulaires sup\'erieures~(\cite{[Se]} \S 2.3).

\subsection{Classification}

La classification suivante regroupe la proposition 15 (\S 2.4) et la partie 2.6 de \cite{[Se]}.

\begin{Ssgrpesmax}
Soit $G$ un sous-groupe de $\GL(V)$.
\begin{enumerate}
\item Si l'ordre de $G$ est premier \`a $p$, alors on est dans l'un des deux cas suivants~:
	\begin{enumerate}
	\item $G$ est un sous-groupe exceptionnel~;
	\item $G$ est contenu dans le normalisateur d'un sous-groupe de Cartan.
	\end{enumerate}
\item Si $p$ divise l'ordre de $G$ alors on est dans l'un des deux cas suivants~:
	\begin{enumerate}
	\item $G$ contient $\SL(V)$~;
	\item $G$ est contenu dans un sous-groupe de Borel.
	\end{enumerate}
\end{enumerate}
\end{Ssgrpesmax}

On recherchera donc des homoth\'eties dans l'image de la repr\'esentation~$\phip$ en distinguant selon les cas de cette classification.

\section{Cas d'une repr\'esentation irr\'eductible}\label{sec:irr}

Dans cette partie, on traite le cas o\`u la repr\'esentation $\phip$ est irr\'eductible~; les r\'esultats obtenus sont r\'esum\'es dans la partie \ref{ssec:ResuIrr}.

Dans la forme finale du th\'eor\`eme de l'introduction, le nombre premier~$p$ est suppos\'e non ramifi\'e dans le corps $K$.
On s'est n\'eanmoins attach\'e \`a \'etablir aussi, dans la mesure du possible, des \'enonc\'es ne d\'ependant que du degr\'e de~$K$ (proposition~\ref{prop:SL},~\ref{prop:HomCdep} et~\ref{prop:HomCnondep}, corollaire~\ref{cor:excep}) et th\'eor\`eme 1 de la partie~\ref{ssec:ResuIrr}).

\begin{notations}\label{not:introIrr}\ 
\begin{enumerate}
\item On note $e_p$ le plus petit indice de ramification dans l'extension~$K/\Q$ d'une place de~$K$ au-dessus de $p$~; il est born\'e par le degr\'e $d$ du corps $K$. 
\item Le d\'eterminant de la repr\'esentation~$\phip$ est le caract\`ere cyclotomique (\cite{[Se]}, page~273). On note $\delta_p$ l'ordre du sous-groupe de $\Fpx$ image de $G_p$ par le d\'eterminant~;~$\delta_p$ est \'egal au degr\'e de l'extension~$K(\mu_p)/K$, qui est aussi celui de l'extension~$\Q(\mu_p) / (\Q(\mu_p) \cap K)$ ($\mu_p$ d\'esignant l'ensembles des racines $p$-i\`emes de l'unit\'e dans $\C$). Ainsi,~$\delta_p$ vaut~$p-1$ si et seulement si les corps $K$ et $\Q(\mu_p)$ sont lin\'eairement disjoints~; l'extension~$\Q(\mu_p) / \Q$ \'etant totalement ramifi\'ee en~$p$, c'est notamment le cas si le nombre premier~$p$ est non ramifi\'e dans~$K$. En g\'en\'eral,~$\delta_p$ est divisible par~$\frac{p-1}{\mathrm{pgcd}(e_p, p-1)}$~; il poss\`ede donc un diviseur sup\'erieur ou \'egal \`a~$\frac{p-1}{d}$.
\end{enumerate}
\end{notations}

\subsection{Sous-groupe d'ordre divisible par $p$}

\begin{proposition}\label{prop:SL}
On suppose que~$G_p$ contient~$\SL(E_p)$. Alors $G_p$ contient l'image r\'eciproque par le d\'eterminant de $\det (G_p)$~; en particulier, $G_p$ contient le sous-groupe des homoth\'eties d'ordre $2 \mathrm{pgcd} \left(\delta_p , \frac{p-1}{2}\right)$, qui est sup\'erieur ou \'egal \`a~$\frac{p-1}{d}$.
\end{proposition}

\begin{proof}
Le sous-groupe des homoth\'eties dont le d\'eterminant est dans~$\det (G_p)$ est exactement celui d'ordre~$2 \mathrm{pgcd} \left(\delta_p , \frac{p-1}{2}\right)$, cet ordre \'etant \'egal \`a~$2 \delta_p$ si~$\delta_p$ divise~$\frac{p-1}{2}$ et \`a~$\delta_p$ sinon.
\end{proof}

\begin{corollaire}
On suppose que le d\'eterminant de~$\phip$ est surjectif dans~$\Fpx$  et que~$G_p$ contient~$\SL(E_p)$~; alors la repr\'esentation~$\phip$ est surjective.
\end{corollaire}

\subsection{Sous-groupes exceptionnels}\label{subsec:excep}

Dans cette partie, on montre que, lorsque~$p$ est strictement sup\'erieur \`a une borne ne d\'ependant que du corps~$K$, l'image~$G_p$ de la repr\'esentation~$\phip$ n'est pas un sous-groupe exceptionnel de~$\GL(E_p)$.
Ce r\'esultat est un analogue du lemme~18 (\S 8.4) de~\cite{[CheEff]} qui est valable sur $\Q$ (les lemmes~26 et~27~(\S 2.1.2)  de~\cite{[Eck05]} pr\'esentent une discussion similaire~; voir aussi dans \cite{[Se]} les arguments de la proposition~17~(\S2.7) et de~\S4.2b)).

\begin{proposition}\label{prop:excep}
On suppose que $p$ est strictement sup\'erieur \`a $20e_p + 1$ ; alors~$G_p$ n'est pas un sous-groupe exceptionnel de~$\GL(E_p)$.
\end{proposition}

La d\'emonstration de la proposition \ref{prop:excep} utilise le lemme suivant (analogue au lemme~18' (\S8.4) de~\cite{[CheEff]}).

\begin{lemme}\label{lem:excep}
L'image de $G_p$ dans $\PGL(E_p)$ contient un \'el\'ement d'ordre sup\'erieur ou \'egal \`a $\frac{p-1}{4e_p}$.
\end{lemme}

\begin{proof}
On raisonne comme dans la d\'emonstration du lemme~18'~(\S8.4) de~\cite{[CheEff]}.

Soit $\p$ une place de $K$ au-dessus de~$p$ dont l'indice de ramification dans~$K/\Q$ est \'egal \`a~$e_p$. On note~$\Kp$ le compl\'et\'e de~$K$ en~$\p$ et on fixe une place de~$\aK$ au-dessus de~$\p$~; ce choix donne un plongement du groupe de Galois absolu~$G_{\Kp}$ de $\Kp$ dans $G_K$. On trouvera l'\'el\'ement d'ordre recherch\'e dans l'image par $\phip$ du sous-groupe d'inertie de~$G_{\Kp}$~;
on raisonne pour cela selon le type de r\'eduction semi-stable de $E$ en $\p$.

On suppose d'abord que l'invariant $j$ de $E$ n'est pas entier en $\p$. Alors $E$ a potentiellement r\'eduction multiplicative en $\p$ et il existe une extension $\Kp'$ de $\Kp$, de degr\'e inf\'erieur ou \'egal \`a $2$, sur laquelle $E$ devient une courbe de Tate (\cite{[Sil]} appendice C th\'eor\`eme 14.1).
 D'apr\`es le \S1.12 de \cite{[Se]}, il existe alors une droite~$W$ de~$E_p$ stable par le sous-groupe d'inertie~$I_{\Kp'}$ de~$G_{\Kp'}$~;
  l'action de~$I_{\Kp'}$  sur~$W$ est donn\'ee par le caract\`ere fondamental de niveau~$1$ du groupe d'inertie mod\'er\'ee de~$\Kp'$ \'elev\'e \`a une puissance \'egale au degr\'e de ramification de~$\Kp'$ sur~$\Qp$ (celui-ci divisant $2e_p$)~;   l'action de~$I_{\Kp'}$  sur le quotient de~$E_p$ par~$W$ est triviale.
   Le caract\`ere fondamental de niveau~$1$ de l'inertie mod\'er\'ee de~$\Kp'$ \'etant surjectif dans $\Fpx$, on obtient que $G_p$ contient la puissance $2e_p$-i\`eme d'un demi-sous-groupe de Cartan d\'eploy\'e (voir partie \ref{sssec:descr_Cdep})~;
   l'image dans~$\PGL(E_p)$ d'un tel sous-groupe contient un \'el\'ement d'ordre~$\frac{p-1}{\mathrm{pgcd}(2e_p , p-1)}$, qui est sup\'erieur ou \'egal \`a~$\frac{p-1}{4e_p}$.

On suppose ensuite que l'invariant $j$ de $E$ est entier en $\p$.
Alors il existe une extension~$\Kp'$ de~$\Kp$ sur laquelle~$E$ a bonne r\'eduction.
On peut choisir~$\Kp'$ d'indice de ramification sur~$\Kp$ divisant l'ordre du groupe des automorphismes de la courbe elliptique  obtenue par r\'eduction au corps r\'esiduel de~$\Kp'$~(\cite{[SeTa]} \S 2, d\'emonstration du th\'eor\`eme~2)~;
 comme~$p$ est suppos\'e sup\'erieur ou \'egal \`a~$5$,  ce groupe est cyclique d'ordre~$2$,~$4$~ou~$6$.
  Quitte \`a tordre la courbe~$E$ par un caract\`ere quadratique (ce qui ne change pas l'image de $G_p$ dans $\PGL(E_p)$), on peut se ramener d'un indice de ramification $6$ \`a $3$~; ainsi, on peut supposer que l'indice de ramification de $\Kp'$ sur $\Kp$ est inf\'erieur ou \'egal \`a $4$~; on le note~$\ep$. On raisonne ensuite selon les propri\'et\'es du groupe formel associ\'e \`a $E$ sur $\Kp'$.

Si ce groupe formel est de hauteur~$1$ (r\'eduction ordinaire), alors, comme dans le cas de r\'eduction multiplicative (\cite{[Se]} proposition 11, \S1.11; \cite{[Eck05]} lemme 22, \S 2.1.1), l'image par $\phip$ du sous-groupe d'inertie de~$G_{\Kp'}$ contient la puissance $\ep e_p$-i\`eme d'un demi-sous-groupe de Cartan d\'eploy\'e~; son image dans~$\PGL(E_p)$ contient un \'el\'ement d'ordre $\frac{p-1}{\mathrm{pgcd}(\ep e_p , p-1)}$.
 
Si le groupe formel est de hauteur $2$ (r\'eduction supersinguli\`ere) et que le polygone de Newton de la multiplication par~$p$ a deux pentes distinctes, alors le sous-groupe d'inertie de~$G_{\Kp'}$ contient un \'el\'ement d'ordre $p$~(voir \cite{[CheEff]} \S 8.4 Lemme 18' b2 et le lemme 25 (\S 2.1.1) de~\cite{[Eck05]})~; il en est donc de m\^eme pour son image dans~$\PGL(E_p)$.

Enfin, si le groupe formel est de hauteur $2$ et que le polygone de Newton de la multiplication par $p$ a une seule pente, alors l'image par~$\phip$ du sous-groupe d'inertie de~$G_{\Kp'}$ contient la puissance~$\ep e_p$-i\`eme d'un sous-groupe de Cartan non d\'eploy\'e (voir \cite{[CheEff]} \S 8.4 Lemme 18' b3, \S 1.10 de \cite{[Se]} et le lemme 23 (\S 2.1.1) de~\cite{[Eck05]})~; son image dans~$\PGL(E_p)$ contient alors un \'el\'ement d'ordre~$\frac{p + 1}{\mathrm{pgcd}(\ep e_p , p + 1)}$.

Dans ces trois cas, on constate que l'image de $G_p$ dans $\PGL(E_p)$ contient un \'el\'ement d'ordre sup\'erieur ou \'egal \`a $\frac{p-1}{4e_p}$.
\end{proof}

\begin{proof}(de la proposition)

On remarque que les groupes $\mathfrak{A}_4$, $\mathfrak{A}_5$ et $\mathfrak{S}_4$ sont form\'es d'\'el\'ements d'ordre inf\'erieur ou \'egal \`a $5$. Or, pour tout $p$ strictement sup\'erieur \`a $20e_p + 1$, on a $\frac{p-1}{4e_p}$ strictement sup\'erieur \`a $5$. Le lemme \ref{lem:excep} permet donc de conclure.

\end{proof}

\begin{corollaire}\label{cor:excep}
On suppose que $p$ est strictement sup\'erieur \`a $20d + 1$. Alors~$G_p$ n'est pas un sous-groupe exceptionnel de~$\GL(E_p)$.
\end{corollaire}

\begin{remarque}\label{rem:excepnonram}
S'il existe une place de $K$ au-dessus de $p$ non ramifi\'ee dans $K/\Q$ (c'est-\`a-dire, que l'indice $e_p$ vaut $1$),
 la discussion plus pr\'ecise du lemme~18 (\S8.4) de \cite{[CheEff]}, impliquant la connaissance des plongements de $\mathfrak{A}_4$, $\mathfrak{A}_5$ et $\mathfrak{S}_4$ dans $\PGL_2(\Fp)$, montre que l'hypoth\`ese du corollaire \ref{cor:excep} peut \^etre remplac\'ee par \og $p$ sup\'erieur ou \'egal \`a $17$ \fg.
\end{remarque}

\subsection{Normalisateur d'un sous-groupe de Cartan}\label{subsec:NCartan}

Dans cette partie, on suppose que l'image~$G_p$ de la repr\'esentation~$\phip$ est contenue dans le normalisateur d'un sous-groupe de Cartan, d\'eploy\'e ou non. Le groupe~$G_p$ est alors d'ordre premier \`a~$p$ et la repr\'esentation~$\phip$ est irr\'eductible si et seulement si elle n'est pas diagonalisable, c'est-\`a-dire que~$G_p$ n'est contenu dans aucun sous-groupe de Cartan d\'eploy\'e de~$\GL(E_p)$.

On fera en fait ici une hypoth\`ese plus faible, en supposant que~$G_p$ est contenu soit dans le normalisateur d'un sous-groupe de Cartan non d\'eploy\'e,  soit dans le normalisateur d'un sous-groupe de Cartan d\'eploy\'e mais pas dans le sous-groupe de Cartan d\'eploy\'e lui-m\^eme.
Ce cadre couvre tous les cas irr\'eductibles, mais aussi certaines r\'eductibles, de la repr\'esentation~$\phip$~; il suffit n\'eanmoins pour obtenir que~$G_p$ contient le sous-groupe des homoth\'eties form\'e du carr\'e de $\det(G_p)$.

\begin{proposition}\label{prop:HomCdep}
On suppose qu'il existe un sous-groupe de Cartan d\'eploy\'e de~$\GL(E_p)$ ne contenant pas~$G_p$, mais dont le normalisateur contient~$G_p$~; alors $G_p$ contient le sous-groupe des homoth\'eties d'ordre~$\frac{\delta_p}{\mathrm{pgcd}(2 , \delta_p)}$  (qui est sup\'erieur ou \'egal \`a~$\frac{p-1}{2d}$).
\end{proposition}
\begin{proof}
Soient~$C_d$ un sous-groupe de Cartan d\'eploy\'e v\'erifiant les hypoth\`eses de la proposition et~$N(C_d)$ son normalisateur (voir la partie  \ref{sssec:descr_Cdep} pour la description de ces sous-groupes).
On se place dans une base de~$E_p$ o\`u~$C_d$ est form\'e des matrices diagonales et~$N(C_d)$ des matrices anti-diagonales.  
 Soit $g$ un \'el\'ement de~$G_p$ dont le d\'eterminant est d'ordre~$\delta_p$ dans~$\Fpx$.

On suppose d'abord que $g$ est dans le sous-groupe de Cartan~$C_d$. Par hypoth\`ese, il existe un \'el\'ement~$h$ de~$G_p$ qui appartient au compl\'ementaire de~$C_d$ dans~$N(C_d)$. La conjugaison par $h$ \'echange alors les deux termes diagonaux de $g$~; on en d\'eduit que l'\'el\'ement~$hgh^{-1}g$ de~$G_p$ est une homoth\'etie de rapport~$\det(g)$.

On suppose ensuite que $g$ est dans~$N(C_d) \backslash C_d$. On observe alors que le carr\'e de~$g$ est une homoth\'etie de rapport $-\det(g)$, contenue dans $G_p$.

On constate que dans les deux cas, $G_p$ contient une homoth\'etie de rapport $(\det(g))^2$, d'o\`u le r\'esultat.
\end{proof}

\begin{proposition}\label{prop:HomCnondep}
On suppose que $G_p$ est inclus dans le normalisateur d'un sous-groupe de Cartan non d\'eploy\'e~;
alors $G_p$ contient le sous-groupe des homoth\'eties d'ordre~$\frac{\delta_p}{\mathrm{pgcd}(2 , \delta_p)}$  (qui est sup\'erieur ou \'egal \`a~$\frac{p-1}{2d}$).

\end{proposition}

\begin{proof}
Soient~$C_{nd}$ un sous-groupe de Cartan non d\'eploy\'e dont le normalisateur contient~$G_p$ et~$N(C_{nd})$ ce normalisateur (voir la partie  \ref{sssec:descr_Cnondep} pour la description de ces sous-groupes).
Soit~$g$ un \'el\'ement de $G_p$ dont le d\'eterminant est d'ordre~$\delta_p$.

On suppose d'abord que~$g$ est dans le sous-groupe de Cartan $C_{nd}$~; on v\'erifie alors que l'\'el\'ement~$g^{p+1}$ de~$G_p$ est une homoth\'etie de rapport~$\det(g)$.

On suppose ensuite que $g$ est dans le compl\'ementaire de $C_{nd}$ dans~$N(C_{nd})$~; on v\'erifie alors que le carr\'e de $g$ est une homoth\'etie de rapport $-\det(g)$, contenue dans~$G_p$.

On constate que dans les deux cas, $G_p$ contient une homoth\'etie de rapport~$(\det(g))^2$, d'o\`u le r\'esultat.
\end{proof}

\begin{corollaire}\label{cor:HomdetsurjC}
On suppose que le d\'eterminant de $\phip$ est surjectif dans~$\Fpx$ et que~$G_p$ est contenu soit dans le normalisateur d'un sous-groupe de Cartan non d\'eploy\'e,  soit dans le normalisateur d'un sous-groupe de Cartan d\'eploy\'e mais pas dans le sous-groupe de Cartan d\'eploy\'e lui-m\^eme~; alors~$G_p$ contient tous les carr\'es  des homoth\'eties.
\end{corollaire}

\begin{remarque}\label{rem:Cartan_p=1mod4}
Soit~$\gamma_p$ un \'el\'ement de~$\Fpx$ d'ordre~$\delta_p$.
D'apr\`es les d\'emonstrations des propositions~\ref{prop:HomCdep} et~\ref{prop:HomCnondep}, le groupe~$G_p$ contient une homoth\'etie de rapport~$\gamma_p$ ou une homoth\'etie de rapport~$-\gamma_p$. Or, l'\'el\'ement~$- \gamma_p$ de~$\Fpx$ est d'ordre~:
\begin{itemize}
\item[$\bullet$] $2\delta_p$ si~$\delta_p$ est impair~;
\item[$\bullet$] $\frac{\delta_p}{2}$ si~$\delta_p$ est divisible par~$2$ mais pas par~$4$~;
\item[$\bullet$] $\delta_p$ si~$\delta_p$ est divisible par $4$.
\end{itemize}
On en d\'eduit que, dans le cadre du corollaire (o\`u~$\delta_p$ vaut~$p-1$), si~$p$ est congru \`a~$1$ modulo~$4$, alors $G_p$ contient toutes les homoth\'eties.
\end{remarque}

\subsection{R\'esum\'e dans le cas irr\'eductible et th\'eor\`eme sur $\Q$}\label{ssec:ResuIrr}

On peut rassembler les r\'esultats sur les cas irr\'eductibles de la repr\'esentation $\phip$ en divers \'enonc\'es. J'en ai retenu ici deux versions~: dans la premi\`ere, la borne sur~$p$ ne d\'epend que du degr\'e du corps~$K$, mais l'indice du sous-groupe des homoth\'eties contenu dans~$G_p$ d\'epend aussi de~$K$ (toujours par son degr\'e)~; dans la seconde, la borne sur~$p$ d\'epend du discriminant de~$K$, mais on obtient que~$G_p$ contient tous les carr\'es des homoth\'eties.

\begin{TheoHommodpA_Irr1}
On suppose la repr\'esentation~$\phip$ irr\'eductible. Si~$p$ est strictement sup\'erieur \`a~$20 d + 1$, alors~$G_p$ contient un sous-groupe des homoth\'eties d'ordre sup\'erieur ou \'egal \`a $\frac{p-1}{2d}$.
\end{TheoHommodpA_Irr1}

\begin{proof}
Par le corollaire \ref{cor:excep}, l'hypoth\`ese faite sur $p$ garantit que $G_p$ n'est pas un sous-groupe exceptionnel. Les propositions~\ref{prop:SL},~\ref{prop:HomCdep} et~\ref{prop:HomCnondep} garantissent alors que~$G_p$ contient le sous-groupe des homoth\'eties d'ordre~$\frac{\delta_p}{\mathrm{pgcd}(2, \delta_p)}$~; cet ordre est sup\'erieur ou \'egal \`a $\frac{p-1}{2d}$.
\end{proof}

\begin{TheoHommodpA_Irr2}
On suppose la repr\'esentation~$\phip$ irr\'eductible. Si~$p$ est non ramifi\'e dans~$K$ et sup\'erieur ou \'egal \`a~$17$, alors~$G_p$ contient les carr\'es des homoth\'eties. Si de plus~$p$ est congru \`a~$1$ modulo~$4$, alors~$G_p$ contient toutes les homoth\'eties.
\end{TheoHommodpA_Irr2}

\begin{proof}
Comme $p$ est non ramifi\'e dans $K$, le d\'eterminant de $\phip$ est surjectif dans $\Fpx$ (voir notation \ref{not:introIrr}), donc $\delta_p$ vaut $p-1$.

Si l'ordre de $G_p$ est divisible par $p$, alors $G_p$ contient $\SL(E_p)$ et $\phip$ est surjective.
La remarque \ref{rem:excepnonram} assure ensuite que, $p$ \'etant sup\'erieur ou \'egal \`a $17$, $G_p$ n'est pas un sous-groupe exceptionnel.
Enfin, le corollaire \ref{cor:HomdetsurjC} implique que $G_p$ contient les carr\'es des homoth\'eties dans les cas irr\'eductibles restants.
La derni\`ere assertion r\'esulte de la remarque \ref{rem:Cartan_p=1mod4}.
\end{proof}

Dans \cite{[Maz]} (th\'eor\`eme 1 de l'introduction), Mazur a montr\'e que, lorsque le corps de base est celui des rationnels~$\Q$ et que~$p$ n'est pas dans l'ensemble $\{2, 3, 5, 7, 13, 11, 17, 19, 37, 43, 67, 163 \}$, alors la repr\'esentation~$\phip$ est irr\'eductible.
Les cas trait\'es pr\'ec\'edemment suffisent donc pour \'enoncer le r\'esultat sur les homoth\'eties contenues dans l'image de $\phip$ pour le corps des rationnels, comme annonc\'e dans l'introduction.

\begin{TheoHommodpAQ}
On suppose que le corps $K$ est \'egal \`a $\Q$. Si $p$ est sup\'erieur ou \'egal \`a $23$ et diff\'erent de $37$, $43$, $67$ et $163$, alors~$G_p$ contient tous les carr\'es des homoth\'eties.
Si de plus $p$ est congru \`a $1$ modulo $4$, alors $G_p$ contient toutes les homoth\'eties.
\end{TheoHommodpAQ}

\section{Cas d'une repr\'esentation r\'eductible}\label{sec:red}

On suppose dans tout cette partie que la repr\'esentation $\phip$ est r\'eductible, c'est-\`a-dire que son image est contenue dans un sous-groupe de Borel de~$\GL(E_p)$.

La courbe~$E$ poss\`ede alors un sous-groupe d'ordre~$p$ globalement stable par l'action du groupe de Galois~$G_K$.
On fixe un tel sous-groupe~$W$~; il lui est associ\'e une isog\'enie de~$E$ de degr\'e~$p$, d\'efinie sur~$K$.
L'action de~$G_K$ sur~$W(\aK)$ est donn\'ee par un caract\`ere continu de~$G_K$ dans $\Fpx$~; on le note $\la$ et, suivant la terminologie introduite dans  \cite{[Maz]}, on l'appelle le caract\`ere d'isog\'enie associ\'e au couple~$(E,W)$.
On note~$\kip$ le caract\`ere cyclotomique, qui est aussi le d\'eterminant de~$\phip$ (\cite{[Se]}, page~273).
On fixe \'egalement une base de $E_p$ dont le premier vecteur engendre~$W(\aK)$~; dans cette base la matrice de la repr\'esentation~$\phip$ est triangulaire sup\'erieure, de termes diagonaux~$(\la, \kip\la^{-1})$.

Le th\'eor\`eme de l'introduction  pour une repr\'esentation~$\phip$ r\'eductible est d\'emontr\'e dans la partie~\ref{ssec:thmAred}. Le raisonnement employ\'e suit l'\'etude du caract\`ere d'isog\'enie initi\'ee par Mazur dans~\cite{[Maz]} (\S 5 et 6) et poursuivie par Momose dans~\cite{[Mom]} (voir aussi~\cite{[Krau2]},~\cite{[Krau]} et~\cite{[Bil]} pour des discussions similaires).
Les parties~\ref{ssec:Inerenp} et~\ref{subsec:horsp} d\'ecrivent les propri\'et\'es locales de~$\la$, \`a savoir ses restrictions \`a des sous-groupes d'inertie ou de d\'ecomposition pour les places finies de~$K$.
Dans la partie~\ref{ssec:thmAred} la th\'eorie du corps de classes globale, appliqu\'ee au caract\`ere~$\lad$, permet d'expliciter les situations pr\'esentant une obstruction \`a la pr\'esence d'un \og gros \fg\ sous-groupe des homoth\'eties dans l'image de~$\phip$. Pour lever une de ces obstructions, on utilisera notamment les bornes uniformes pour l'ordre des points de torsion des courbes elliptiques en fonction du degr\'e du corps de base (\cite{[Pa]}, \cite{[Mer96]}).

Les r\'esultats des parties~\ref{ssec:Inerenp} et~\ref{subsec:horsp} sont d\'ej\`a mentionn\'es et utilis\'es dans~\cite{[Mom]} (\S 2).
Pour la commodit\'e du lecteur, on en a rappel\'e ici des \'enonc\'es adapt\'es \`a leur utilisation dans la partie~\ref{ssec:thmAred} et quelques \'el\'ements de d\'emonstration (voir aussi la partie~1 de~\cite{[Dav08]}).

\begin{hypotheses}\ 

\begin{enumerate}
\item Lorsque le corps de base est celui des rationnels, Mazur a montr\'e (\cite{[Maz]}) que, si~$p$ n'est pas dans l'ensemble $\{2, 3, 5, 7, 13, 11, 17, 19, 37, 43, 67, 163 \}$, alors la repr\'esentation~$\phip$ est irr\'eductible. Dans toute cette partie, on supposera donc que le corps~$K$ est diff\'erent de~$\Q$.
\item Dans toute cette partie, on suppose \'egalement que le nombre premier~$p$ est non ramifi\'e dans l'extension~$K/\Q$.
\end{enumerate}
\end{hypotheses}

\subsection{Action des sous-groupes d'inertie des places au-dessus de $p$}\label{ssec:Inerenp}

Soit~$\p$ un id\'eal premier de~$K$ situ\'e au-dessus de~$p$.
On fixe~$D_{\p}$ un sous-groupe de d\'ecomposition pour~$\p$ dans~$G_K$ et on note~$\Ip$ son sous-groupe d'inertie.

La proposition suivante g\'en\'eralise la proposition 5.1 (\S 5) de \cite{[Maz]}~; elle figure d\'ej\`a, sous forme de remarque (remarque 1 de la partie 2), dans \cite{[Mom]}.

\begin{proposition}\label{coeffap}
Il existe un entier $\ap$ valant $0$, $4$, $6$, $8$ ou $12$ tel que la restriction au sous-groupe d'inertie~$\Ip$ de la puissance douzi\`eme du caract\`ere d'isog\'enie $\la$ soit le caract\`ere cyclotomique \'elev\'e \`a la puissance $\ap$.
\end{proposition}
\begin{proof}

On note~$\Kp$ le compl\'et\'e de~$K$ en~$\p$ ; on en fixe une cl\^oture alg\'ebrique~$\aKp$ et on note~$\Kp^{nr}$ son extension non ramifi\'ee maximale.
On distingue le raisonnement selon si l'invariant $j$ de $E$ est entier ou non en $\p$.

On suppose d'abord que cet invariant~$j$ n'est pas entier en~$\p$.
Alors~$E$ a potentiellement r\'eduction multiplicative en~$\p$ et il existe une extension~$\Kp'$ de~$\Kp$, de degr\'e au plus~$2$, sur laquelle~$E$ est isomorphe \`a une courbe de Tate (\cite{[Sil]} appendice C th\'eor\`eme 14.1).
On note $D'_{\p}$ le sous-groupe de $D_{\p}$ correspondant au groupe de Galois absolu de $\Kp'$ et $\Ip'$ son sous-groupe d'inertie.

D'apr\`es \cite{[Se]} (proposition 13 de \S1.12 et page 273 de \S 1.11), le caract\`ere~$\la$ restreint au sous-groupe $\Ip'$ est soit trivial soit \'egal au caract\`ere cyclotomique $\kip$.
Comme $\Ip'$ est un sous-groupe d'indice au plus $2$ de $\Ip$, on obtient que $\la_{|\Ip}^2$ est soit trivial soit \'egal \`a~$\kip^2$. Finalement ~$\la_{|\Ip}^{12}$ co\"incide avec~$\kip^{\ap}$ pour~$\ap$ valant~$0$ ou~$12$.
 
On suppose ensuite que l'invariant~$j$ de~$E$ est entier en~$\p$, c'est-\`a-dire que~$E$ a potentiellement bonne r\'eduction en~$\p$.
On note~$\theta_{p-1}$ le caract\`ere fondamental de niveau~$1$ du groupe d'inertie mod\'er\'ee~$\Ip^t$ de~$\Kp$ (voir \cite{[Se]} page 267)~;~$p$ \'etant suppos\'e non ramifi\'e dans l'extension~$K/\Q$,~$\theta_{p-1}$ induit le caract\`ere cyclotomique sur le groupe d'inertie~$\Ip$.
Le caract\`ere~$\la$ restreint \`a~$\Ip$, continu et \`a valeurs dans~$\Fpx$, se factorise par un caract\`ere de~$\Ip^t$, encore \`a valeurs dans $\Fpx$.
D'apr\`es la proposition 5~(\S1.7) de \cite{[Se]}, un tel caract\`ere est de la forme $\theta_{p-1}^{a'_{\p}}$, pour un entier~$a'_{\p}$.

Il existe  une extension $\Kp'$ de $\Kp$ sur laquelle $E$ a bonne r\'eduction~; comme dans la d\'emonstration du lemme  \ref{lem:excep}, on peut choisir~$\Kp'$ d'indice de ramification sur~$\Kp$ divisant~$2$,~$4$ ou~$6$.
On note $\ep$ cet indice~;~$p$ \'etant suppos\'e non ramifi\'e dans~$K$,~$\ep$ est \'egalement l'indice de ramification absolu de~$\Kp'$.
On note \'egalement~$D'_{\p}$ le sous-groupe de~$D_{\p}$ isomorphe au groupe de Galois absolu de~$\Kp'$,~$\Ip'$ son sous-groupe d'inertie,~$\Ip'^{t}$ son groupe d'inertie mod\'er\'ee et~$\theta'_{p-1}$ le caract\`ere fondamental de niveau~$1$ de~$\Ip'^{t}$.

Comme~$E$ a bonne r\'eduction sur~$\Kp'$, le corollaire~3.4.4 de \cite{[Ray]} (ou le \S 1.13 de \cite{[Se]}) donne qu'il existe un entier~$r_{\p}$, compris entre~$0$ et~$\ep$, tel que le caract\`ere~$\la$ restreint \`a~$\Ip'$ se factorise par le caract\`ere~$(\theta'_{p-1})^{r_{\p}}$ de~$\Ip'^{t}$.
Or le groupe d'inertie mod\'er\'ee~$\Ip'^{t}$ s'identifie \`a un sous-groupe de~$\Ip^{t}$ sur lequel le caract\`ere fondamental de niveau~$1$ de~$\Kp$ induit la puissance~$\ep$-i\`eme du caract\`ere fondamental de niveau $1$ de $\Kp'$ (\cite{[Se]} \S1.4). On obtient ainsi sur~$\Ip'^{t}$~:
$$
(\theta'_{p-1})^{r_{\p}} = (\theta_{p-1})^{a'_{\p}} = (\theta'_{p-1})^{\ep a'_{\p}}.
$$
Comme le caract\`ere~$\theta'_{p-1}$ est d'ordre~$p-1$, on obtient la congruence~: 
$$
(\star)\  \ep\ap' \equiv \rp \mod p-1.
$$
Comme~$\ep$ est un diviseur de~$4$ ou~$6$, il divise~$12$. On a donc :
$$
12\ap' \equiv \frac{12}{\ep}\ep\ap' \equiv \frac{12}{\ep}\rp \mod p-1.
$$
On a d\'eduit qu'on a sur $\Ip$~: 
$$
\la^{12}_{|\Ip} = \kip^{\frac{12}{\ep}\rp}.
$$

On est donc ramen\'e \`a d\'eterminer, pour chacune des valeurs possibles du couple $(\ep,\rp)$ (elles sont en nombre fini), la valeur de l'entier $\ap =  \frac{12}{\ep}\rp$ (modulo~$p-1$).
On remarque que pour certaines valeurs du couple~$(\ep,\rp)$, la congruence~$(\star)$ n'a pas de solution en~$\ap'$~: c'est le cas lorsque~$\ep$ est pair et~$\rp$ impair. Pour d'autres valeurs de~$(\ep,\rp)$, l'existence d'un entier~$\ap'$ satisfaisant la congruence $(\star)$ donne une condition suppl\'ementaire sur $p$~:
\begin{itemize}
\item si~$3$ divise~$\ep$ mais ne divise pas~$\rp$, alors~$3$ ne divise pas~$p-1$, donc~$p$ est congru \`a~$2$ modulo~$3$ ;
\item si~$\ep$ est \'egal \`a~$4$ mais ne divise pas~$\rp$, alors~$4$ ne divise pas~$p-1$ donc~$p$ est congru \`a~$3$ modulo~$4$.
\end{itemize}
Le tableau suivant r\'esume les valeurs possibles pour le couple~$(\ep,\rp)$, la valeur de~$\ap =  \frac{12}{\ep}\rp$ correspondante, ainsi que les informations suppl\'ementaires obtenues. Les cases marqu\'ees d'une croix correspondent aux cas o\`u la congruence~$(\star)$ n'a pas de solutions.
$$
\begin{array}{| c | c | c | c | c | c | c | c | c | c | }
\hline
\ep & \multicolumn{2}{c |}{1} & \multicolumn{3}{c |}{2}  & \multicolumn{4}{c |}{3} \\ \hline
\rp & 0 & 1 & 0 & 1 & 2 & 0 & 1 & 2 & 3 \\ \hline
\ap = \frac{12}{\ep}\rp & 0 & 12 & 0 & \times & 12 & 0 & 4 & 8 & 12 \\ \hline
p & \multicolumn{2}{c |}{-} & \multicolumn{3}{c |}{-}  & - & \multicolumn{2}{c |}{ p \equiv 2 [3]} & - \\ \hline
\end{array}
$$
$$
\begin{array}{ | c | c | c | c | c | c | c | c | c | c | c | c | c |}
\hline
\ep &  \multicolumn{5}{c |}{4} & \multicolumn{7}{c |}{6} \\ \hline
\rp &  0 & 1 & 2 & 3 & 4 & 0 & 1 & 2 & 3 & 4 & 5 & 6 \\ \hline
\ap = \frac{12}{\ep}\rp & 0 & \times & 6 & \times & 12 & 0 & \times & 4 & \times & 8 & \times & 12 \\ \hline
p & - & \times & p \equiv 3 [4] & \times & - & -  & \times & p \equiv 2 [3] &  \times & p \equiv 2 [3] & \times & - \\ \hline
\end{array}
$$
\end{proof}

\subsection{Places hors de $p$}\label{subsec:horsp}

Soit $\q$ un id\'eal premier de $K$ qui n'est pas au-dessus de $p$.
On note :
\begin{itemize}
\item[$\bullet$] $\Kq$ le compl\'et\'e de $K$ en $\q$~;
\item[$\bullet$] $\kq$ le corps r\'esiduel de $K$ en $\q$~;
\item[$\bullet$] $\Dq$ un sous-groupe de d\'ecomposition (fix\'e) pour $\q$ dans $G_K$ ;
\item[$\bullet$] $\Iq$ le sous-groupe d'inertie de $\Dq$ ;
\item[$\bullet$] $\sq$ un rel\`evement (fix\'e) dans $\Dq$ du frobenius de~$\kq$~;
\item[$\bullet$] $N\q$ la norme de $\q$ dans l'extension $K/\Q$ ; sa valeur absolue est le cardinal de~$\kq$.

\end{itemize}
Le but de cette partie est de d\'eterminer les images par le caract\`ere d'isog\'enie~$\la$ de~$\sq$ et~$\Iq$.

\subsubsection{Invariant $j$ non entier}

On suppose dans cette partie que $E$ a potentiellement r\'eduction multiplicative en $\q$.
\begin{proposition}\label{RMhorsp}
Si $E$ a potentiellement r\'eduction multiplicative en $\q$, alors 
\begin{enumerate}
\item $\la^2$ est non ramifi\'e en $\q$ ;
\item $\la^2(\sq)$ vaut $1$ ou $(N\q)^2$ modulo $p$.
\end{enumerate}
\end{proposition}

\begin{proof}
Il existe une extension~$\Kq'$ de~$\Kq$, de degr\'e au plus~$2$, sur laquelle~$E$ est isomorphe \`a une courbe de Tate (\cite{[Sil]} appendice C th\'eor\`eme 14.1).

Le sous-groupe de d\'ecomposition $\Dq$ est naturellement isomorphe au groupe de Galois absolu de $\Kq$~; on note~$\Dq'$ son sous-groupe correspondant au groupe de Galois absolu de~$\Kq'$ et~$\Iq'$ le sous-groupe d'inertie de $\Dq'$.
Le groupe $\Dq'$  (respectivement $\Iq'$) est d'indice au plus $2$ dans $\Dq$ (respectivement $\Iq$).
Le corps r\'esiduel~$\kq'$ de~$\Kq'$ est une extension de degr\'e inf\'erieur ou \'egal \`a~$2$ de~$\kq$ ; on fixe un \'el\'ement~$\sq'$ de~$\Dq'$ qui induit sur une cl\^oture alg\'ebrique~$\akq$ de~$\kq$ le morphisme de Frobenius de~$\kq'$.

Soit~$\overline{\Kq'}$ une cl\^oture alg\'ebrique de~$\Kq'$~; l'ensemble~$E(\overline{\Kq'})[p]$ des points de~$p$-torsion de~$E$ sur~$\overline{\Kq'}$ est d\'ecrit par la suite exacte de~$\Dq'$-modules suivante (voir \cite{[Se]} \S 1.12 ou \cite{[Abladic]} IV.A.1.2)~:
$$
0 \longrightarrow \mu_p(\overline{\Kq'}) \longrightarrow E(\overline{\Kq'})[p] \longrightarrow \Z / p\Z \longrightarrow 0,
$$
o\`u $\mu_p(\overline{\Kq'})$ d\'esigne le groupe des racines $p$-i\`emes de l'unit\'e de $\overline{\Kq'}$ et $\Dq'$ agit trivialement sur $\Z / p\Z$. 
Deux cas sont possibles pour~$W(\overline{\Kq'})$, qui est un sous-$\Dq'$-module d'ordre~$p$ de~$E(\overline{\Kq'})[p]$.

Soit l'intersection~$\mu_p(\overline{\Kq'}) \cap W(\overline{\Kq'})$ est r\'eduite \`a l'\'el\'ement neutre.
Alors~$W(\overline{\Kq'})$ s'envoie isomorphiquement (comme~$\Dq'$-module) dans $\Z / p\Z$.
Ceci implique que~$\Dq'$ agit trivialement sur~$W(\overline{\Kq'})$ ; il en est donc de m\^eme pour $\Iq'$ et $\sq'$, qui ont ainsi une image triviale par $\la$.

Soit $\mu_p(\overline{\Kq'})$ et $W(\overline{\Kq'})$ co\"incident.
Alors $\Iq'$ agit trivialement sur~$W(\overline{\Kq'})$ et~$\sq'$ agit sur~$W(\overline{\Kq'})$ comme le frobenius de~$\kq'$ agit sur les racines~$p$-i\`emes de l'unit\'e de~$\overline{\kq'}$ ;
ceci donne~:~$\la(\sq') = |\kq' | \Mod p$.

Dans les deux cas, le caract\`ere~$\la$ est trivial sur le sous-groupe~$\Iq'$, qui est d'indice au plus~$2$ dans~$\Iq$ ; on en d\'eduit que~$\la^2$ est trivial sur~$\Iq$. Pour d\'eterminer la valeur de~$\la(\sq)$, on raisonne selon les diff\'erentes possibilit\'es pour l'extension~$\kq' / \kq$.

Si les corps r\'esiduels~$\kq'$ et~$\kq$ co\"incident, alors~$\sq^2$ est un \'el\'ement de~$\Dq'$ qui induit sur~$\akq$ le carr\'e du morphisme de Frobenius de~$\kq'$ ; l'\'el\'ement $(\sq'\sq^{-1})^2$ est donc dans~$\Iq'$~; on en d\'eduit qu'on~a~:
$$
\la(\sq^2) = \la(\sq'^2) = |\kq' |^2 \Mod p =  |\kq |^2 \Mod p = (N\q)^2 \Mod p.
$$

Si le corps~$\kq'$ est de degr\'e~$2$ sur~$\kq$, alors~$\sq^2$ est un \'el\'ement de~$\Dq'$ qui induit sur~$\akq$ le morphisme de Frobenius de~$\kq'$ ; l'\'el\'ement~$\sq'\sq^{-2}$ est donc dans~$\Iq'$~; on en d\'eduit qu'on~a~:
$$
\la(\sq^2) = \la(\sq') = |\kq' | \Mod p =  |\kq |^2 \Mod p = (N\q)^2 \Mod p.
$$
\end{proof}

\subsubsection{Invariant $j$ entier}
On suppose dans cette partie que~$E$ a potentiellement bonne r\'eduction en~$\q$.

On note~$\Kl$ l'extension galoisienne de~$K$ trivialisant le caract\`ere d'isog\'enie~$\la$~; elle est cyclique et de degr\'e divisant~$p-1$.

L'existence d'un point d'ordre~$p$ (choisi sup\'erieur ou \'egal \`a~$5$) d\'efini sur le corps~$\Kl$ assure le lemme suivant.

\begin{lemme}\label{lem:BR}
En toute place de $\Kl$ au-dessus de $\q$, $E$ n'a pas r\'eduction additive.
\end{lemme}

\begin{proof} 
On rappelle ici rapidement les id\'ees de la d\'emonstration, qui proviennent du~\S 6 de \cite{[Maz]} .

Soient~$\q^{\la}$ une place de~$\Kl$ au-dessus de~$\q$~; on note~$\Kq'$ le compl\'et\'e de~$\Kl$ en~$\q^{\la}$ et~$\kq'$ son corps r\'esiduel.
On suppose par l'absurde que~$E$ a r\'eduction additive sur~$\Kq'$.

 Soit~$\mathcal{E}_{\Kq'}$ le mod\`ele de N\'eron de~$E$ sur~$\Kq'$,~$\widetilde{\mathcal{E}}_{\Kq'}$ sa fibre sp\'eciale et~$\widetilde{\mathcal{E}}_{\Kq'}^{0}$ la composante neutre de cette fibre sp\'eciale. Alors~:
\begin{enumerate}
 	\item[(i)]le groupe~$\widetilde{\mathcal{E}}_{\Kq'}^{0}(\kq')$ est isomorphe \`a $(\kq',+)$~;
	 \item[(ii)]le quotient~$\widetilde{\mathcal{E}}_{\Kq'}(\kq')/\widetilde{\mathcal{E}}_{\Kq'}^{0}(\kq')$ est un groupe d'ordre au plus~$4$ (\cite{[Sil]} appendice~C \S15 tableau 15.1 ou \cite{[Mf]} tableau page 46)~;
 	\item[(iii)]on dispose d'une application de r\'eduction de~$E(\Kq')$ dans~$\widetilde{\mathcal{E}}_{\Kq'}(\kq')$ qui est un morphisme de groupes.
\end{enumerate}
Par d\'efinition de~$\Kq'$, le groupe~$E(\Kq')$ poss\`ede un point~$P$ d'ordre~$p$. On raisonne selon l'image~$\widetilde{P}$ de~$P$ dans~$\widetilde{\mathcal{E}}_{\Kq'}(\kq')$ par l'application de r\'eduction.

 Soit $\widetilde{P}$ appartient \`a la composante neutre de $\widetilde{\mathcal{E}}_{\Kq'}(\kq')$.
 Alors~$P$ appartient au sous-groupe~$E^{0}(\Kq')$ de~$E(\Kq')$ form\'e des points envoy\'es dans~$\widetilde{\mathcal{E}}_{\Kq'}^{0}(\kq')$ par l'application de r\'eduction.
 Ce sous-groupe est d\'ecrit par la suite exacte de groupes suivante (\cite{[Sil]} VII \S 2 propositions 2.1 et 2.2) :
 $$
 0 \longrightarrow \widehat{E}_{\Kq'}(\mathfrak{m}_{\Kq'}) \longrightarrow E^{0}(\Kq') \longrightarrow (\kq' , +) \longrightarrow 0,
 $$
 o\`u~$\widehat{E}_{\Kq'}$ d\'esigne le groupe formel associ\'e \`a une \'equation de Weierstrass minimale de~$E$ sur l'anneau des entiers~$\mathcal{O}_{\Kq'}$ de~$\Kq'$ et~$\mathfrak{m}_{\Kq'}$ l'unique id\'eal maximal de~$\mathcal{O}_{\Kq'}$.
 
 Comme~$p$ est diff\'erent de la caract\'eristique r\'esiduelle de~$\Kq'$, le groupe~$\widehat{E}_{\Kq'}(\mathfrak{m}_{\Kq'})$ n'a pas de sous-groupe d'ordre~$p$ (\cite{[Sil]} IV \S 3 proposition 3.2) 
 Ceci implique que le sous-groupe engendr\'e par~$P$ est envoy\'e injectivement dans~$(\kq', +)$.
 Or le groupe~$(\kq', +)$ n'a pas non plus de point d'ordre~$p$. On obtient donc une contradiction.

 Soit $\widetilde{P}$ n'appartient pas \`a~$\widetilde{\mathcal{E}}_{\Kq'}^{0}(\kq')$. Alors~$\widetilde{P}$ engendre dans~$\widetilde{\mathcal{E}}_{\Kq'}(\kq')$ un sous-groupe d'ordre~$p$ d'intersection triviale avec~$\widetilde{\mathcal{E}}_{\Kq'}^{0}(\kq')$.
 Ce sous-groupe s'envoie donc injectivement dans le quotient~$\widetilde{\mathcal{E}}_{\Kq'}(\kq')/\widetilde{\mathcal{E}}_{\Kq'}^{0}(\kq')$. Ce quotient \'etant d'ordre au plus~$4$, on obtient une contradiction avec le choix de~$p$ sup\'erieur ou \'egal \`a~$5$.
\end{proof}

\begin{proposition}\label{BRhorsp:iner}
Si $E$ a potentiellement bonne r\'eduction en $\q$, alors l'ordre de l'image du sous-groupe d'inertie~$\Iq$ par le caract\`ere d'isog\'enie~$\la$ divise~$4$ ou~$6$~; en particulier, le caract\`ere~$\la^{12}$ est non ramifi\'e en~$\q$.
\end{proposition}

\begin{proof}
L'image de~$\Iq$ par~$\la$ est un sous-groupe de~$\Fpx$ ; il est donc cyclique (et son ordre divise~$p-1$).
D'autre part,~$\la(\Iq)$ est isomorphe au sous-groupe d'inertie de l'extension~$\Kl / K$ en la place~$\q$ (canoniquement d\'efini car l'extension~$\Kl / K$ est ab\'elienne).
Soient~$\q^{\la}$ une place de~$\Kl$ au-dessus de~$\q$,~$\Kq'$ le compl\'et\'e de~$\Kl$ en~$\q^{\la}$ et~$\Kq^{nr}$ l'extension maximale non ramifi\'ee de~$\Kq$.
Alors~$\la(\Iq)$ est aussi isomorphe au groupe de Galois de l'extension~$\Kq'\Kq^{nr} /  \Kq^{nr}$. D'apr\`es le lemme~\ref{lem:BR},~$E$ a bonne r\'eduction sur le corps $\Kq'\Kq^{nr}$.

Soit~$K^{\phip}$ l'extension galoisienne de $K$ trivialisant la repr\'esentation~$\phip$~; elle contient~$\Kl$. Soient~$\q^{\phip}$ une place de~$K^{\phip}$  au-dessus de~$\q^{\la}$ et $\Kq''$ le compl\'et\'e de~$K^{\phip}$ en $\q^{\phip}$~; il contient naturellement~$\Kq'$. Ainsi, le corps~$\Kq''\Kq^{nr}$ contient~$\Kq'\Kq^{nr}$.

Par ailleurs,~$p$ \'etant sup\'erieur ou \'egal \`a~$3$ et diff\'erent de la caract\'eristique r\'esiduelle de~$\q$,~$\Kq''\Kq^{nr}$ est la plus petite extension de~$\Kq^{nr}$ sur laquelle~$E$ acquiert bonne r\'eduction (\cite{[SeTa]} \S2 corollaire 3 du th\'eor\`eme 2). Comme~$E$ a bonne r\'eduction sur~$\Kq'\Kq^{nr}$, on obtient que~$\Kq''\Kq^{nr}$ est inclus dans~$\Kq'\Kq^{nr}$. On en d\'eduit finalement que les corps~$\Kq''\Kq^{nr}$ et~$\Kq'\Kq^{nr}$ co\"incident.

D'apr\`es la d\'emonstration du th\'eor\`eme~2~(\S2) de~\cite{[SeTa]} , le groupe de Galois de l'extension~$\Kq''\Kq^{nr} / \Kq^{nr}$ s'injecte dans le groupe des automorphismes de la courbe elliptique obtenue \`a partir de~$E \times_{K} \Kq''$ par r\'eduction modulo~$\q^{\phip}$.
Les groupes pouvant survenir comme groupes d'automorphismes d'une telle courbe elliptique sont connus (\cite{[Sil]} appendice A, proposition 1.2 et exercice A.1)~: ils sont soit cyclique d'ordre~$2$,~$4$ ou~$6$, soit produit semi-direct d'un groupe cyclique d'ordre~$3$ par un groupe cyclique d'ordre~$4$, soit produit semi-direct du groupe  des quaternions (d'ordre~$8$) par un groupe cyclique d'ordre~$3$. On v\'erifie alors que les sous-groupes cycliques de tels groupes sont d'ordre divisant $4$ ou $6$.
\end{proof}

\begin{proposition}\label{BRhorsp:frob}
Si~$E$ a potentiellement bonne r\'eduction en~$\q$, alors~:
\begin{enumerate}
\item L'action de~$\sq$ sur le module de Tate en~$p$ de~$E$ a un polyn\^ome caract\'eristique~$\Pq(X)$ \`a coefficients dans~$\Z$ (et ind\'ependant de~$p$).
\item Les racines de~$\Pq(X)$ dans~$\C$ sont conjugu\'ees complexes l'une de l'autre et ont m\^eme valeur absolue~$\sqrt{ | N\q | }$ ; elles engendrent dans~$\C$ un corps~$\Lq$ qui est soit~$\Q$ soit un corps quadratique imaginaire.
\item Pour tout  id\'eal premier~$\Pcq$ de~$\Lq$ au-dessus de~$p$, les images des racines de~$\Pq$ dans~$\OLq / \Pcq$ sont dans~$\Fpx$ ; il existe une racine~$\bq$ de~$\Pq(X)$ telle que~$\la(\sq)$ soit \'egal \`a $ \bq \Mod \Pcq$.
\end{enumerate}
\end{proposition}

\begin{proof}
Les points 1 et 2 r\'esultent du th\'eor\`eme~3~(\S 2) de~\cite{[SeTa]}. Comme~$G_p$ est inclus dans un sous-groupe de Borel, la r\'eduction du polyn\^ome~$\Pq$ modulo~$p$ est scind\'ee dans~$\Fp$ et ses racines sont les r\'eductions modulo~$\Pcq$ des racines de~$\Pq$ dans~$\Lq$ ; ceci donne le point 3.
\end{proof}

\subsection{Th\'eor\`eme dans le cas r\'eductible}\label{ssec:thmAred}

L'\'etude men\'ee dans les parties~\ref{ssec:Inerenp} et~\ref{subsec:horsp} permet de d\'emontrer le th\'eor\`eme de l'introduction lorsque la repr\'esentation $\phip$ est r\'eductible~: on obtient une borne d\'ependant du discriminant, du degr\'e et du nombre de classes d'id\'eaux du corps de base $K$.

\begin{TheoHommodpA_Red}
On suppose la repr\'esentation~$\phip$ r\'eductible. Si~$p$ est non ramifi\'e dans~$K$ et strictement sup\'erieur \`a~$(1 + 3^{6dh})^2$, alors~$G_p$ contient un sous-groupe d'indice divisant~$8$ ou~$12$ des homoth\'eties.
\end{TheoHommodpA_Red}

On verra que les homoth\'eties recherch\'ees se trouvent dans le sous-groupe engendr\'e par les images par~$\phip$ des sous-groupes d'inertie de~$G_K$ associ\'es aux diff\'erentes places de~$K$ au-dessus de~$p$~; ces images sont notamment d\'ecrites par la famille de coefficients~$(\ap)_{\p | p}$ introduite dans la partie~\ref{ssec:Inerenp} (proposition~\ref{coeffap}).
La d\'emonstration du th\'eor\`eme se d\'ecompose ainsi en quatre lemmes, qui correspondent aux diff\'erentes familles~$(\ap)_{\p | p}$ possibles.

On constate qu'il n'y a que deux situations pr\'esentant une obstruction \`a ce que $G_p$ contienne un \og gros \fg\  sous-groupe des homoth\'eties (lemmes \ref{ap=0ou12} et \ref{ap=4ou8}).
La th\'eorie du corps de classes globale appliqu\'ee \`a l'extension ab\'elienne de K d\'efinie par la puissance douzi\`eme~$\la^{12}$ du caract\`ere d'isog\'enie permet de borner les nombres premiers~$p$ pour lesquels ces obstructions surviennent.
Dans un cas (lemme \ref{ap=4ou8}), elle est lev\'ee en \'etudiant la compatibilit\'e entre les actions sur le sous-groupe d'isog\'enie des sous-groupes d'inertie au-dessus de~$p$ d'une part et d'un morphisme de Frobenius pour une place de $K$ au-dessus de~$2$ d'autre part.
L'autre cas (lemme \ref{ap=0ou12}) provient de courbes elliptiques ayant un point d'ordre~$p$ d\'efini sur une \og petite \fg\ extension de~$K$~: on utilise alors les bornes uniformes sur l'ordre des points de torsion des courbes elliptiques qui figurent dans~\cite{[Mer96]} ou~\cite{[Pa]}.

On rappelle qu'on a suppos\'e, d\`es le d\'ebut de la partie~\ref{sec:red},~$p$ non ramifi\'e dans~$K$~; cette hypoth\`ese implique en particulier que le caract\`ere cyclotomique~$\kip$ est surjectif de~$G_K$ dans~$\Fpx$, car sa restriction \`a tout sous-groupe d'inertie de~$G_K$ en une place de~$K$ au-dessus de~$p$ l'est.

On rappelle \'egalement qu'on a fix\'e au d\'ebut de la partie~\ref{sec:red} une base de~$E_p$ dans laquelle la forme matricielle de~$\phip$  est triangulaire sup\'erieure avec pour caract\`eres diagonaux~$(\la, \kip\la^{-1})$. On remarque qu'on a, pour tout couple~$(\alpha, \beta)$ d'\'el\'ements de~$\Fp$~:
$$
\begin{pmatrix}
\alpha & \beta \\
0 & \alpha
\end{pmatrix}^p
= 
\begin{pmatrix}
\alpha^p & p\alpha^{p-1}\beta \\
0 & \alpha^p
\end{pmatrix}
=
\begin{pmatrix}
\alpha & 0 \\
0 & \alpha
\end{pmatrix}.
$$
Dans chacun des lemmes ci-dessous, on recherchera donc dans~$G_p$ un \'el\'ement dont la diagonale est une homoth\'etie de rapport engendrant un sous-groupe d'indice divisant~$8$ ou~$12$ de~$\Fpx$.

\begin{lemme}\label{lem:ap=6} 
On suppose qu'il existe un id\'eal premier~$\p$ de~$K$ au-dessus de~$p$ pour lequel le coefficient~$\ap$ est \'egal \`a~$6$~; alors $G_p$ contient les carr\'es des homoth\'eties.
\end{lemme}

\begin{proof}
Soient~$\p$ un id\'eal premier de~$K$ au-dessus de~$p$ pour lequel le coefficient~$\ap$ est \'egal \`a~$6$ et~$\Ip$ un sous-groupe d'inertie pour $\p$ dans $G_K$.

D'apr\`es la d\'emonstration de la proposition~\ref{coeffap},~$p$ est congru \`a $3$ modulo $4$ et la restriction \`a~$\Ip$ du caract\`ere d'isog\'enie~$\la$  co\"incide avec~$\kip^{\ap'}$, pour un entier~$\ap'$ v\'erifiant~:~$4\ap'  \equiv 2 \Mod p -1$.
Il existe donc un entier relatif~$u$ v\'erifiant~:~$4\ap'  = 2 + u (p -1)$.
Comme l'entier~$\frac{p-1}{2}$ est impair, on obtient que $u$ est \'egalement impair et finalement on a~:~$2\ap'  \equiv 1 + \frac{p -1}{2} \Mod p -1$.
On en d\'eduit qu'on a~:
$$
\la^2_{|\Ip} = \kip^{2\ap'} = \kip\kip^{\frac{p-1}{2}}.
$$
Ainsi, la restriction \`a~$\Ip$ du carr\'e de la repr\'esentation~$\phip$ est de la forme
$$
\varphi_{p|\Ip}^2
=
\begin{pmatrix}
\la^2 & \star \\
0 & \left(\kip\la^{-1}\right)^2
\end{pmatrix}
=
\begin{pmatrix}
\kip\kip^{\frac{p-1}{2}}& \star \\
0 & \kip\kip^{\frac{p-1}{2}}
\end{pmatrix}.
$$

Comme~$p$ est congru \`a~$3$ modulo~$4$ et que~$\chi_{p|\Ip}$ est surjectif dans~$\Fpx$, l'image de $\Ip$ par le caract\`ere $\kip\kip^{\frac{p-1}{2}}$ est form\'ee de tous les carr\'es de $\Fpx$.
\end{proof}

\begin{lemme}\label{ap=0ou12}
On suppose qu'on est dans l'un des deux cas suivants~:
\begin{itemize}
\item[$\bullet$] pour tout id\'eal premier~$\p$ de~$K$ au-dessus de~$p$, le coefficient~$\ap$ est \'egal \`a~$0$~; 
\item[$\bullet$] pour tout id\'eal premier~$\p$ de~$K$ au-dessus de~$p$, le coefficient~$\ap$ est \'egal \`a~$12$.
\end{itemize}
Alors~$p$ est inf\'erieur ou \'egal \`a~$(1 + 3^{6dh})^2$.
\end{lemme}

\begin{proof}
On suppose d'abord que tous les coefficients $\ap$ sont \'egaux \`a $0$.

D'apr\`es les propositions \ref{RMhorsp} et \ref{BRhorsp:iner}, le caract\`ere $\la^{12}$ est non ramifi\'e hors de $p$.
Par d\'efinition des coefficients $\ap$ (proposition~\ref{coeffap}), $\la^{12}$ est aussi non ramifi\'e aux places au-dessus de $p$.
Ainsi, l'extension~$K^{\lad}$ de~$K$ trivialisant~$\lad$, qui est ab\'elienne, est non ramifi\'ee en toute place finie de~$K$. Le corps~$K^{\lad}$ est donc inclus dans le corps de classes de Hilbert de~$K$~; son degr\'e (sur~$\Q$) est donc inf\'erieur ou \'egal \`a~$dh$.
Comme le caract\`ere~$\la$ est d'image cyclique, le corps~$\Kl$ qui trivialise~$\la$ est une extension de degr\'e divisant~$12$ de~$K^{\lad}$.
Son degr\'e (sur~$\Q$) est ainsi inf\'erieur ou \'egal \`a~$12dh$.

Or, la courbe~$E$ poss\`ede un point d'ordre~$p$ d\'efini sur~$\Kl$.
Il existe alors plusieurs bornes pour~$p$ en fonction du degr\'e de~$\Kl$ sur~$\Q$.
On a choisi pour l'expression du lemme celle, obtenue par Merel et Oesterl\'e, qui figure dans les introductions de~\cite{[Pa]} et~\cite{[Mer96]}.

Lorsque tous les coefficients $\ap$ sont \'egaux \`a $12$, on consid\`ere le quotient de~$E$  par le sous-groupe d'isog\'enie~$W$.

 On obtient une courbe elliptique~$E'$ d\'efinie sur~$K$, isog\`ene \`a~$E$~sur~$K$~; la courbe~$E'$ poss\`ede~$E[p]/W$ comme sous-groupe d'ordre~$p$ d\'efini sur~$K$ et le caract\`ere d'isog\'enie associ\'e \`a ce sous-groupe d'ordre~$p$ est~$\kip\la^{-1}$. Le raisonnement pr\'ec\'edent appliqu\'e au couple~$(E' , E[p]/W)$ donne alors la m\^eme borne pour~$p$.
\end{proof}

\begin{lemme}\label{ap=4ou8}
On suppose qu'on est dans l'un des deux cas suivants~:
\begin{itemize}
\item[$\bullet$] pour tout id\'eal premier~$\p$ de~$K$ au-dessus de~$p$, le coefficient~$\ap$ est \'egal \`a~$4$~; 
\item[$\bullet$] pour tout id\'eal premier~$\p$ de~$K$ au-dessus de~$p$, le coefficient~$\ap$ est \'egal \`a~$8$.
\end{itemize}
Alors $p$ est inf\'erieur ou \'egal \`a $\left( 2^{6dh} + 2^{4dh} \right)^2$.
\end{lemme}

\begin{proof}
On suppose d'abord que tous les coefficients $\ap$ sont \'egaux \`a $4$.
Alors le caract\`ere $\la^{12}\kip^{-4}$ est non ramifi\'e en toute place finie de~$K$~; on a donc~$\la^{12h} = \kip^{4h}$.
On raisonne ensuite selon le type de r\'eduction de la courbe~$E$ en une place fix\'ee de~$K$ au-dessus de~$2$.
Soit~$\mathcal{L}$  une telle place et~$N\mathcal{L}$ sa norme dans l'extension~$K/\Q$ (en particulier,~$N\mathcal{L}$ est de valeur absolue inf\'erieure ou \'egale \`a~$2^d$).
Avec les notations de la partie~\ref{subsec:horsp},~on a~$\la^{12h}(\sigma_{\mathcal{L}}) = (N\mathcal{L})^{4h} \Mod p$.

Si~$E$ a potentiellement r\'eduction multiplicative en~$\mathcal{L}$, alors d'apr\`es la proposition~\ref{RMhorsp}, on a~$\la^{12}(\sigma_{\mathcal{L}}) = 1 \Mod p$
 ou~$\la^{12}(\sigma_{\mathcal{L}}) = (N\mathcal{L})^{12} \Mod p$.
 On en d\'eduit que~$p$ divise l'un des deux entiers relatifs~$(N\mathcal{L})^{4h} - 1$ ou~$(N\mathcal{L})^{12h} - (N\mathcal{L})^{4h}$~;
  ainsi,~$p$ divise~$(N\mathcal{L})^{4h} - 1$ ou~$(N\mathcal{L})^{8h} - 1$.
  Comme ces deux entiers sont strictement positifs, on obtient que~$p$ est inf\'erieur ou \'egal \`a leur maximum, qui est lui-m\^eme inf\'erieur ou \'egal \`a~$2^{8dh} - 1$.

Si $E$ a potentiellement bonne r\'eduction  en $\mathcal{L}$, alors d'apr\`es la proposition \ref{BRhorsp:frob} (et avec les notations comme dans cette proposition), on a~: $\la^{12}(\sigma_{\mathcal{L}}) = \beta_{\mathcal{L}}^{12} \Mod \mathcal{P}_{\mathcal{L}}$. On a donc dans $L_{\mathcal{L}}$~: $\beta_{\mathcal{L}}^{12h} - (N\mathcal{L})^{4h}  \equiv 0 \Mod \mathcal{P}_{\mathcal{L}}$.
Comme les entiers alg\'ebriques~$\beta_{\mathcal{L}}^{12h}$ et~$(N\mathcal{L})^{4h}$  ont des valeurs absolues complexes diff\'erentes (respectivement~$(N\mathcal{L})^{6h}$ et~$(N\mathcal{L})^{4h}$), ils ne sont pas \'egaux.
On en d\'eduit que~$p$ divise l'entier relatif non nul $N_{L_{\mathcal{L}}/\Q}\left( \beta_{\mathcal{L}}^{12h} - (N\mathcal{L})^{4h} \right)$ dont la valeur absolue est inf\'erieure ou \'egale \`a $\left( (N\mathcal{L})^{6h} + (N\mathcal{L})^{4h} \right)^2$ et finalement \`a~$\left( 2^{6dh} + 2^{4dh} \right)^2$ (qui est strictement sup\'erieur \`a~$2^{8dh} - 1$).

Lorsque tous les coefficients  $\ap$ sont \'egaux \`a $8$, on proc\`ede comme \`a la fin de la d\'emonstration du lemme \ref{ap=0ou12}~: on se ram\`ene \`a une courbe isog\`ene \`a $E$ sur $K$, poss\'edant un sous-groupe d'ordre $p$ d\'efini sur $K$ pour lequel le caract\`ere d'isog\'enie est \'egal \`a~$\kip\la^{-1}$ et on applique le raisonnement pr\'ec\'edent.
 \end{proof}
 
 \begin{remarque}Pour tout r\'eel~$x$ sup\'erieur ou \'egal \`a~$1$, on a
 $$
 \left(1 + \left( 3^6 \right)^x \right)^2 \geq \left( \left( 2^6 \right)^x + \left( 2^4 \right)^x \right)^2.
 $$
 C'est pourquoi la borne du lemme \ref{ap=0ou12} pr\'evaut sur celle du lemme \ref{ap=4ou8} dans l'\'enonc\'e du th\'eor\`eme de l'introduction.
 \end{remarque}
 
 \begin{lemme}
On suppose qu'il existe deux id\'eaux premiers~$\p$ et~$\p'$ de~$K$ au-dessus de~$p$ pour lesquels les coefficients~$\ap$ et~$a_{\p'}$ sont distincts~; alors $G_p$ contient un sous-groupe d'indice divisant~$8$ ou~$12$ des homoth\'eties.
\end{lemme}

\begin{proof}
Soient $\p$ et $\p'$ deux id\'eaux premiers de~$K$ au-dessus de~$p$ pour lesquels les coefficients~$\ap$ et~$a_{\p'}$ sont distincts~; on fixe $\Ip$ et $I_{\p'}$ des sous-groupes d'inertie pour $\p$ et $\p'$ dans $G_K$.

D'apr\`es le lemme \ref{lem:ap=6}, on peut se contenter de traiter le cas o\`u~$\ap$ et~$a_{\p'}$ sont diff\'erents de~$6$. Les paires possibles pour les valeurs de~$\ap$ et~$a_{\p'}$ sont alors~: 
$$
\{0, 4 \}, \{0, 8 \}, \{0, 12 \}, \{ 4, 8 \}, \{4, 12 \}, \{8, 12 \}.
$$ 
On traite chacune de ces valeurs successivement. On note~$x$ un \'el\'ement g\'en\'erateur de~$\Fpx$.

Si la paire $\{ \ap , a_{\p'} \}$ est $\{0, 4\}$, alors la restriction de la puissance douzi\`eme de~$\phip$ \`a~$\Ip$ est de la forme
$$
\varphi^{12}_{|\Ip}
=
\begin{pmatrix}
1 & \star \\
0 & \kip^{12}
\end{pmatrix}
$$
et sa restriction \`a $I_{\p'}$ est de la forme
$$
\varphi^{12}_{| I_{\p'}}
=
\begin{pmatrix}
\kip^{4} & \star \\
0 & \kip^{8}
\end{pmatrix}.
$$

Comme la restriction de~$\kip$ \`a~$\Ip$ (comme \`a~$I_{\p'}$) est surjective dans~$\Fpx$, on obtient que~$G_p$
contient un \'el\'ement de diagonale $(1, x^{-12})$ et un autre de diagonale $((x^3)^4, (x^3)^8)$ ;
 leur produit donne un \'el\'ement de diagonale $(x^{12},x^{12})$. Ainsi,~$G_p$ contient l'homoth\'etie de rapport~$x^{12}$.
 Comme~$p$ est congru \`a~$2$ modulo~$3$ (d\'emonstration de la proposition~\ref{coeffap}), l'\'el\'ement~$x^3$ engendre aussi~$\Fpx$~; on en d\'eduit que~$G_p$ contient les puissances quatri\`emes des homoth\'eties.
 Le cas de la paire $\{8, 12 \}$ se traite de mani\`ere similaire.

Si la paire~$\{ \ap , a_{\p'} \}$ est~$\{0,8 \}$, alors $G_p$ contient un \'el\'ement de diagonale~$(1, x^{12})$ et un autre de diagonale~$((x^3)^8, (x^3)^4)$~; leur produit donne un \'el\'ement de diagonale~$(x^{24},x^{24})$.
 Comme~$p$ est congru \`a~$2$ modulo~$3$,~ $x^3$ engendre aussi~$\Fpx$ donc~$G_p$ contient les puissances huiti\`emes des homoth\'eties. Le cas de la paire $\{4, 12 \}$ se traite de mani\`ere similaire.

Si la paire~$\{ \ap , a_{\p'} \}$ est~$\{0, 12 \}$, alors $G_p$ contient un \'el\'ement de diagonale~$(1, x^{12})$  et un autre de diagonale~$(x^{12}, 1)$ ; leur produit donne un \'el\'ement de diagonale~$(x^{12},x^{12})$.
Dans ce cas $G_p$ contient les puissances douzi\`emes des homoth\'eties.

Si la paire~$\{ \ap , a_{\p'} \}$ est~$\{4, 8 \}$, alors $G_p$ contient un \'el\'ement de diagonale~$(x^4, x^8)$ et un autre de diagonale~$(x^8, x^4)$~; leur produit donne un \'el\'ement de diagonale $(x^{12},x^{12})$.
Comme~$p$ est congru \`a~$2$ modulo~$3$,~$x^3$ engendre aussi~$\Fpx$ donc~$G_p$ contient les puissances quatri\`emes des homoth\'eties.
\end{proof}

\nocite{}
\shorthandoff{:}
\bibliographystyle{alpha}
\bibliography{Hudig.bib}

\end{document}